\documentclass[12pt,a4paper,psamsfonts]{amsart}
\usepackage{amssymb,amscd,amsxtra,calc}
\usepackage{pstricks,pst-node,pst-plot,pst-math} 
\usepackage{cmmib57}

\theoremstyle{plain}
    \newtheorem{thm}{Theorem}[section]
    
    \newtheorem{claim}[thm]{Claim}
    \newtheorem{corollary}[thm]{Corollary}
    \newtheorem{fact}[thm]{Fact}
    \newtheorem{lemma}[thm]{Lemma}
    \newtheorem{proposition}[thm]{Proposition}
    \newtheorem{conjecture}[thm]{Conjecture}
    
    \newtheorem{theorem}[thm]{Theorem}

\theoremstyle{definition}

    \newtheorem{remark}[thm]{Remark}
\theoremstyle{remark}

    \newtheorem{setup}[thm]{}

\newcommand{\C}{\mathbb{C}}

\newcommand{\Q}{\mathbb{Q}}
\newcommand{\PP}{\mathbb{P}}

\newcommand{\OO}{\mathcal{O}}

\newcommand{\aaa}{\operatorname{a}}
\newcommand{\alb}{\operatorname{alb}}
\newcommand{\Alb}{\operatorname{Alb}}

\newcommand{\Exc}{\operatorname{Exc}}

\newcommand{\id}{\operatorname{id}}
\newcommand{\Imm}{\operatorname{Im}}

\newcommand{\Nklt}{\operatorname{Nklt}}

\newcommand{\Supp}{\operatorname{Supp}}

\begin{document}


\title[Ampleness of canonical divisors]
{Ampleness of canonical divisors of hyperbolic normal projective varieties}

\author{Fei Hu}
\address
{
\textsc{Department of Mathematics} \endgraf
\textsc{National University of Singapore, 10 Lower Kent Ridge Road,
Singapore 119076
}}
\email{hf@nus.edu.sg}

\author{Sheng Meng}
\address
{
\textsc{Department of Mathematics} \endgraf
\textsc{National University of Singapore, 10 Lower Kent Ridge Road,
Singapore 119076
}}
\email{ms@nus.edu.sg}

\author{De-Qi Zhang}
\address
{
\textsc{Department of Mathematics} \endgraf
\textsc{National University of Singapore, 10 Lower Kent Ridge Road,
Singapore 119076
}}
\email{matzdq@nus.edu.sg}

\begin{abstract}
Let $X$ be a projective variety
which is algebraic Lang hyperbolic.
We show that Lang's conjecture holds (one direction only):
$X$ and all its subvarieties are of general type
and the canonical divisor $K_X$ is ample at smooth points and Kawamata log terminal points of $X$,
provided that $K_X$ is $\Q$-Cartier,
no Calabi-Yau variety is algebraic Lang hyperbolic
and a weak abundance conjecture holds.
\end{abstract}

\subjclass[2000]
{32Q45, 
14E30}  

\keywords{algebraic Lang hyperbolic variety, ample canonical divisor}

\thanks{}

\maketitle

\section{Introduction}\

We work over the field $\C$ of complex numbers.
A variety
$X$ is {\it Brody hyperbolic} (resp.~{\it algebraic Lang hyperbolic})
if every holomorphic map $V \to X$, where $V$ is the complex line $\C$ (resp.~$V$ is
an abelian variety), is a constant map. Since an abelian variety is a complex torus,
Brody hyperbolicity implies algebraic Lang hyperbolicity.
When $X$ is a compact complex variety, Brody hyperbolicity is equivalent to the
usual Kobayashi hyperbolicity (cf.~\cite{La}).

In the first part (Theorem \ref{ThA} and its consequences \ref{Cor1}, \ref{Cor2}) of this paper,
we let $X$ be a normal projective variety and
aim to show the ampleness of the
{\it canonical divisor $K_X$ of $X$}, assuming that
$X$ is algebraic Lang hyperbolic. We allow $X$ to have arbitrary singularities and
assume only that $X$ is {\it $\Q$-Gorenstein} (so that the ampleness of $K_X$ is well-defined),
i.e., $K_X$ is {\it $\Q$-Cartier:}
$mK_X$ is a Cartier divisor for some positive integer $m$.

For related work,
it was proven in \cite{Pe} that a $3$-dimensional
hyperbolic smooth projective variety $X$ has ample $K_X$ unless $X$ is
a Calabi-Yau manifold where every non-zero effective divisor is ample.
The authors of \cite{HLW} proved the ampleness of $K_X$ when $X$ is a smooth
projective threefold having a K\"ahler metric of negative holomorphic sectional curvature;
they also generalized the results to higher dimensions with some additional conditions.

In the second part of the paper (Theorem \ref{ThC} and its more general form Theorem \ref{ThB}),
we make some contributions toward Lang's conjecture in Corollary \ref{Cor3},
where even the normality of $X$ is not assumed.
Our approach is to take a projective resolution of $X$ and run the relative Minimal Model Program
(MMP) over $X$.
We use only the frame work of MMP, but not its detailed technical part.
Certain mild singularities occur naturally along the way.
See \cite[Definitions 2.34 and 2.37]{KM} for definitions of {\it canonical},
{\it Kawamata log terminal} (klt), and {\it divisorial log terminal} (dlt) {\it singularities}.

In the last part (Proposition \ref{PropA} and its more general form Theorem \ref{ThD}), we try to avoid assuming conjectures.

We now state two conjectures.
Conjecture \ref{nALH} below is long standing.
When $\dim X \le 2$, it is true by the classification of
complex surfaces and the following:

\par \vskip 1pc
{\bf Fact $(*)$}. A (smooth) $K3$ surface has infinitely many
(singular) elliptic curves; see \cite[Theorem in Appendix]{MM} or Proposition \ref{solconj}.

\par \vskip 1pc
In Conjecture \ref{nALH}, the conclusion means the existence of at least one non-constant
holomorphic map $f : V \to X$ from an abelian variety $V$,
but does {\it not} require the image $f(V)$ (or the union of such images) to be Zariski-dense
in $X$. This does not seem sufficient for our purpose to show the non-existence of subvariety $X'$ of Kodaira dimension
zero in an algebraic Lang hyperbolic variety $W$ as in Corollary \ref{Cor3} below
(see \ref{setup1.1}, and think about a proof of the non-hyperbolicity of every normal $K3$ surface
using the Fact $(*)$ above).
Fortunately, we are able to show in Corollary \ref{Cor3}
(or Theorems \ref{ThC} and \ref{ThB}) that the normalization $X$ of $X' \subseteq W$ is a Calabi-Yau variety
and hence $f$ composed with the finite morphism $X \to X' \subseteq W$
produces a non-constant holomorphic map from the abelian variety $V$,
thus deducing a contradiction to the hyperbolicity of $W$.

\begin{conjecture}\label{nALH}
Let $X$ be an absolutely minimal Calabi-Yau variety (cf.~\ref{setup2.1}).
Suppose further that every birational morphism $X \to Y$ onto a normal projective variety is an isomorphism.
Then $X$ is not algebraic Lang hyperbolic.
\end{conjecture}

We need the result below about {\it nef reduction map} and {\it nef dimension}.
A meromorphic map $f : X \dasharrow Y$ between complex varieties is {\it almost holomorphic}
if it is well defined on a Zariski dense open subset $U$ of $X$
and the map $f_{\, | U} : U \to Y$ has compact connected general fibres.

\begin{theorem}\label{Th8aut} (cf.~\cite[Theorem 2.1]{8aut})
Let $L$ be a nef $\Q$-Cartier divisor on a normal projective variety $X$. Then
there exists an almost holomorphic, dominant rational map $f : X \dasharrow Y$ with
connected fibres, called a ``nef reduction map'' such that
\begin{itemize}
\item[(1)] $L$ is numerically trivial on all compact fibres $F$ of $f$ with
$\dim F = \dim X - \dim Y$;
\item[(2)] for every general point $x \in X$ and every irreducible curve $C$ passing
through $x$ with $\dim f(C) > 0$, we have $L . C > 0$.
\end{itemize}
The map $f$ is unique up to birational equivalence of $Y$.
We call $\dim Y$ the ``nef dimension'' of $L$ and
denote it as $n(L)$.
\end{theorem}

\begin{proof}
See \cite{8aut} for the proof.
\end{proof}

Next we state Conjecture \ref{Abund}.
We stress that \ref{Abund} without the extra ``Hyp(A)''
is the usual abundance conjecture and {\it stronger} than our one here.
When $K_X$ is nef and big or when
$\dim X \le 3$, both Conjectures \ref{Abund} (1) and \ref{Abund} (2) (and their log version, even
without the extra Hyp(A)) are true; see \cite[Theorem 3.3, \S 3.13]{KM},
or Proposition \ref{solconj}.

\begin{conjecture}\label{Abund}
Let $X$ be an $n$-dimensional minimal normal projective variety, i.e.,
the canonical divisor $K_X$ is a nef $\Q$-Cartier divisor.
Assume Hyp(A): the nef dimension $n(K_X)$ satisfies $n(K_X) = n$.
\begin{itemize}
\item[(1)]
If $X$ has at worst klt singularities,
then
$K_X$ is semi-ample, i.e.,
the linear system $|mK_X|$ is base-point free for some $m > 0$.
\item[(2)]
If $X$ has at worst canonical singularities and
$K_X \not\equiv 0$ $($not numerically zero$)$, then
the Kodaira dimension $\kappa(X) > 0$.
\end{itemize}
\end{conjecture}

\par \vskip 1pc
Theorems \ref{ThA} and \ref{ThC} below are our main results.
When $X$ has at worst klt singularities, Theorem \ref{ThA} below follows from
the MMP and has been generalized to the quasi-projective case
in \cite{LZ}.
In Theorem \ref{ThA}, we do not impose any condition on the
singularities of $X$, except the $\Q$-Cartierness of $K_X$.
This assumption is necessary to formulate the conclusion
that $K_X$ be ample.
Without assuming Conjecture \ref{nALH} or \ref{Abund} as in Theorem \ref{ThA}, we can
at least say that $K_{X}$ is movable or nef in codimension-one
(cf.~Remark \ref{rThA}). See also Corollaries \ref{Cor1} and \ref{Cor2}
when $\dim X \le 3$.

\begin{theorem}\label{ThA}
Let $X$ be a $\Q$-Gorenstein normal projective variety
which is algebraic Lang hyperbolic.
Assume that Conjecture $\ref{nALH}$ holds for all varieties
birational to $X$, or to any
subvariety of $X$. Further, assume that Conjecture $\ref{Abund} (1)$ holds for
all varieties birational to $X$.

Then $K_X$ is ample at smooth points and klt points of $X$.
To be precise, there is a birational morphism $f_c : X_c \to X$
such that $X_c$ has at worst klt singularities, $K_{X_c}$ is ample, and
$E_c := f_c^*K_X - K_{X_c}$ is an effective and $f_c$-exceptional divisor
with $f_c(E_c) \subseteq \Nklt(X)$, the non-klt locus of $X$.

In particular, $f_c = \id$ and $K_X$
is ample, if $X$ has at worst klt singularities.
\end{theorem}

The normality of $X$ is not assumed in Theorem \ref{ThC} below
which is a special case of Theorem \ref{ThB} by letting $g : X \to W$ there to be
the identity map $\id_X : X \to X$.
When $\dim X \le 3$, Case (3) below does not occur.

\begin{theorem}\label{ThC}
Let $X$ be an algebraic Lang hyperbolic projective variety of dimension $n$.
Assume either $n \le 3$ or Conjecture $\ref{Abund} (2)$
$($resp.~either $n \le 3$ or Conjecture $\ref{Abund} (2)$ without the extra Hyp(A)$)$
holds for all varieties birational to $X$.

Then there is a birational surjective morphism $g_m : X_m \to X$ such that
$X_m$ is a minimal variety with at worst canonical
singularities and one of the following is true.
\begin{itemize}
\item[(1)]
$K_{X_m}$ is ample. Hence both $X_m$ and $X$ are of general type.
\item[(2)]
$g_m : X_m \to X$ is the normalization map. $X_m$ is an absolutely minimal
Calabi-Yau variety with $\dim X_m \ge 3$.
\item[(3)]
There is an almost holomorphic map $\tau : X_m \dasharrow Y$
$($resp.~a holomorphic map $\tau: X_m \to Y)$
such that its general fibre $F$ is an absolutely minimal
Calabi-Yau variety with $3 \le \dim F < \dim X_m$,
and $(g_m)_{|F} : F \to g_m(F) \subset X$ is the normalization map.
\end{itemize}
\end{theorem}

Given a projective variety $W$, let $\widetilde{W} \overset{\sigma}\to W$ be a
projective resolution. We define the {\it albanese variety
of $W$} as $\Alb(W) := \Alb(\widetilde{W})$, which is independent of the choice
of the resolution $\widetilde{W}$, since every two resolutions of $W$ are
dominated by a third one and the albanese variety, being an abelian variety,
contains no rational curves. We define
the {\it albanese $($rational$)$ map} $\alb_W : W \dasharrow \Alb(W)$
as the composition
$$W \overset{\sigma^{-1}}\dasharrow \widetilde{W}
\overset{\alb_{\widetilde{W}}}\longrightarrow \Alb(\widetilde{W}) .$$

One direction of Lang's \cite[Conjecture 5.6]{La}
follows from Conjectures \ref{nALH} and \ref{Abund} (2). See
Remark \ref{rThA} (6) for the other direction.

\begin{corollary}\label{Cor3}
Let $W$ be an algebraic Lang hyperbolic projective variety of dimension $n$.
Assume either $n \le 3$ or Conjecture $\ref{Abund} (2)$ holds for all varieties of dimension $\le n$.
Then we have:
\begin{itemize}
\item[(1)]
If Conjecture $\ref{nALH}$ holds for all varieties of dimension $\le n$,
then $W$ and all its subvarieties are of general type.
\item[(2)]
If the albanese map  $\alb_W : W \dasharrow \Alb(W)$
has general fibres
of dimension $\le 2$, then $W$ is of general type.
\end{itemize}
\end{corollary}

Without assuming Conjecture \ref{Abund} (or \ref{nALH}), we have the following
(see Theorem \ref{ThD} for a generalization). For a singular projective variety $Z$, we define
the {\it Kodaira dimension} $\kappa(Z)$ as $\kappa(\widetilde{Z})$ (cf.~\cite[\S 7.73]{KM})
for some (or equivalently any) projective resolution $\widetilde{Z} \to Z$.

\begin{proposition}\label{PropA}
Let $X$ be an algebraic Lang hyperbolic projective variety.
Assume one of the following conditions.
\begin{itemize}
\item[(i)]
$X$ has maximal albanese dimension, i.e.,
the albanese map $\alb_X : X \dasharrow \Alb(X)$
is generically finite $($but not necessarily surjective$)$.
\item[(ii)]
The Kodaira dimension $\kappa(X) \ge \dim X - 2$.
\item[(iii)]
$\kappa(X) \ge \dim X - 3$, and Conjecture $\ref{nALH}$ holds in dimension three.
\end{itemize}
Then $X$ is of general type.
\end{proposition}

\begin{remark}\label{rThA}
(1) In Theorem \ref{ThA},
by the equality $f_c^*K_X = K_{X_c} + E_c$ and the ampleness of $K_{X_c}$, the {\it exceptional locus} $\Exc(f_c)$ (the subset of $X_c$ along which
$f_c$ is not isomorphic) is contained in $\Supp E_c$. Indeed, if $C$ is an $f_c$-contractible curve, then
$0 = C . f_c^*K_X = C . K_{X_c} + C . E_c > C . E_c$, so $C \subseteq \Supp E_c$.
This and the effectivity of $E_c$ justifies the
phrasal: $K_{X}$ is ample outside $f_c(E_c)$.

(2) Without assuming Conjecture \ref{nALH} or \ref{Abund},
the proof of Theorem \ref{ThA} (Claim \ref{nef} and the equality (\ref{EqC2}) above it)
shows that $(f')^*K_X = K_{X'} + E'$ with $K_{X'}$ nef and $E'  \ge 0$ $f'$-exceptional.
Hence $K_X = f'_*K_{X'}$ is movable, or nef in codimension-one.

(3) Let $X_2 \to X_1$ be a finite morphism (but not necessarily surjective).
If $X_1$ is Brody hyperbolic or algebraic Lang hyperbolic then so is $X_2$.
The converse is not true.

(4) Every algebraic Lang hyperbolic projective variety $X_1$ is absolutely minimal
in the sense of \ref{setup2.1}, i.e.,
every birational map $h: X_2 \dasharrow X_1$ from
a normal projective variety $X_2$ with at worst klt singularities, is a well defined morphism.
This result was proved by S.~Kobayashi when $X_2$
is nonsingular. Indeed, let $Z$ be a resolution of the graph of $h$ such that
we have a birational surjective morphisms $p_i : Z \to X_i$
satisfying $h \circ p_2 = p_1$. Then every fibre $p_2^{-1} (x_2)$ is rationally chain connected
by \cite[Corollary 1.5]{HM} and hence $p_1(p_2^{-1} (x_2))$ is a point since
hyperbolic $X_1$ contains no rational curve.
Thus $h$ can be extended to a well defined morphism by \cite[Proof of Lemma 14]{Ka81},
noting that $X_2$ is normal and $p_2$ is surjective, and using the Stein factorization.

(5) If $Y$ is an algebraic Lang hyperbolic Calabi-Yau variety (like $X_m$ and $F$ in
Theorem \ref{ThC} (2) and (3), respectively), then
every birational morphism $h: Y \to Z$ onto a normal projective variety is an isomorphism.
Indeed, by \cite[Corollary 1.5]{Ka88}, $Z$ has only
canonical singularities. Thus the exceptional locus $\Exc(h)$ is covered by
rational curves by \cite[Corollary 1.5]{HM}. Since $Y$ is hyperbolic and hence has no rational
curve, we have $\Exc(h) = \emptyset$ and hence $h$ is an isomorphism, $Z$ being normal
and by the Stein factorization.

(6) Consider the converse of Corollary \ref{Cor3}, i.e., the other direction of Lang
\cite[Conjecture 5.6]{La}, but with the assumption that every non-uniruled projective variety
has a minimal model with at worst canonical singularities
and that abundance Conjecture \ref{Abund} (2) holds.
To be precise, supposing that a projective variety $W$ and all its subvarieties are of
general type, we see that $W$ is algebraic Lang hyperbolic. Indeed, let $f: V \to W$
be a morphism from an abelian variety $V$
and let $V \to X \to f(V)$ be its Stein factorization,
where $V \to X$ has a connected general fibre $F$ and $X \to f(V)$
is a finite morphism. Since $V$ is non-uniruled, so is $F$. Hence $\kappa(F) \ge 0$ by the assumption.
The assumption and Iitaka's $C_{n,m}$ also
imply that $0 = \kappa(V) \ge \kappa(F) + \kappa(X) \ge \kappa(X) \ge \kappa(f(V)) = \dim f(V)$
(cf.~\cite[Corollary 1.2]{Ka85}). Hence $f$ is a constant map.
\end{remark}

\begin{setup}\label{setup1.1}
{\bf Comments about the proofs.}
{\rm
In our proofs, neither the existence of minimal model nor the termination
of MMP is assumed. Let $W$ be an algebraic Lang hyperbolic projective variety.
To show that every subvariety $X$ of $W$ is of general type,
one key observation is the existence of a birational model $X'$ of $X$ with
$K_{X'}$ relative nef over $W$, by using
the main Theorem $1.2$ in \cite{BCHM}.
{\it $K_{X'}$ is indeed nef} since $W$ is hyperbolic (cf.~Lemma \ref{rnef} or \ref{genfin}).
One natural approach is to take a general fibre $F$ (which may not even be normal)
of an Iitaka (rational)
fibration of $X$ (assuming $\kappa(X) \ge 0$)
and prove that $F$ has a minimal model $F_m$. Next, one tries to show that $q(F_m) = 0$
and $F_m$ is a Calabi-Yau variety and then tries to use Conjecture \ref{nALH} to produce
a non-hyperbolic subvariety $S$ of $F_m$, but this does not guarantee the
same on $F \subset X$ (to contradict the hyperbolicity of $X$)
because such $S \subseteq F_m$ might be contracted to a point on $F$.
In our approach, we are able to show that the {\it normalization of $F$ is a Calabi-Yau variety,
which is the key of the proofs.} It would not help even if one assumes the smoothness of the ambient space
$W$ since its subvariety $X$ may not be smooth, or at least normal or Cohen-Macaulay to define
the canonical divisor $K_X$ meaningfully to pull back or push forward.
}
\end{setup}

\par \vskip 1pc \noindent
{\bf Acknowledgement.}
We would like to thank the referee for very careful reading, many suggestions to
improve the paper and the insistence on clarity of exposition.
The last named-author is partially supported by an ARF of NUS.

\section{Preliminary results}

\begin{setup}\label{setup2.1}
{\bf Convention, notation and terminology}

{\it In this paper, by hyperbolic we mean algebraic Lang hyperbolic}.
{\rm

\begin{itemize}
\item[(i)]
We use the notation and terminology in the book of Hartshorne and the book \cite{KM}.

\item[(ii)]
Given two morphisms $g_i : Y_i \to Z$ ($i = 1, 2$) between varieties, a rational map $Y_1 \dasharrow Y_2$
is said to be {\it a map over} $Z$, if the composition $Y_1 \dasharrow Y_2 \overset{g_2}\to Z$
coincides with $g_1 : Y_1 \to Z$.

\item[(iii)]
For a rational map $h : X \dasharrow Y$, we take a birational resolution
$\pi : W \to X$ of the indeterminacy of $h$
such that the composition $h \circ \pi$ is a well defined morphism: $h_1 : W \to Y$.
For a point $y \in Y$, we defined the {\it fibre} $h^{-1} (y)$ as $\pi(h_1^{-1} (y))$.
This definition does not depend on the choice of the resolution $\pi$ of $h$, since every two such resolutions are dominated by a third one.

\item[(iv)]
For a singular projective variety $Z$, we define
the {\it Kodaira dimension} $\kappa(Z)$ as $\kappa(\widetilde{Z})$ (cf.~\cite[\S 7.73]{KM})
for some (or equivalently any) projective resolution $\widetilde{Z} \to Z$.
When $\kappa(\widetilde{Z}) \ge 0$, there is a {\it $($rational$)$ Iitaka fibration},
unique up to birational equivalence,
$I_{Z} : Z \dasharrow Y$ such that its very general fibre
$F$ has $\kappa(F) = 0$ and that $\dim Y = \kappa(Z)$.

\item[(v)]
For two Weil $\Q$-divisors $D_i$ on a normal variety $X$, if $m(D_1 - D_2) \sim 0$ (linear equivalence)
for some integer $m > 0$, we say that $D_1$ and $D_2$ are $\Q$-linearly equivalent
and denote this relation as: $D_1 \sim_{\Q} D_2$.
\item[(vi)]
Let $X$ be a normal projective variety.
$X$ is a {\it Calabi-Yau variety} if
$X$ has at worst canonical singularities,
its canonical divisor is $\Q$-linearly equivalent to zero: $K_X \sim_{\Q} 0$,
and the {\it irregularity} $q(X) := h^1(X, \OO_X) = 0$.
If this is the case, $X$ has Kodaira dimension $\kappa(X) = 0$.

\item[(vii)]
A projective variety
$X$ is {\it of general type} if some (equivalently every) projective resolution $X'$ of
$X$ has maximal Kodaira dimension: $\kappa(X') = \dim X'$.

\item[(viii)]
A $\Q$-Gorenstein variety $X$ is {\it minimal} if the canonical divisor $K_X$ is {\it numerically
effective} (= {\it nef}).
A projective variety $X_1$ is {\it absolutely minimal} if
every birational map $h: X_2 \dasharrow X_1$ from
a normal projective variety $X_2$ with at worst klt singularities, is a well defined morphism.

\end{itemize}
}
\end{setup}

\begin{proposition}\label{solconj}
\begin{itemize}
\item[(1)]
Let $X$ be a projective surface. Then either it has infinitely many rational curves or elliptic curves,
or it is of general type, or it is birational to a simple abelian surface.
\item[(2)]
Let $Y$ be a normal projective surface such that $K_Y \sim_{\Q} 0$ and $Y$ is birational to an abelian surface $A$.
Then $Y$ is isomorphic to $A$.
\item[(3)]
Let $Z$ be a normal projective surface with $K_Z \sim_{\Q} 0$. Then $Z$ is not algebraic Lang hyperbolic.
In particular, Conjecture \ref{nALH} holds when dimension $\le 2$.
\item[(4)]
In dimension $\le 3$, both Conjectures \ref{Abund} (1) and (2) even without the extra Hyp(A) (and even for
log canonical pairs) hold.
\item[(5)]
Both Conjectures \ref{Abund} (1) and (2) even without the extra Hyp(A) hold for varieties of general type.
\item[(6)]
Let $X$ be a variety with maximal albanese dimension, i.e., $\dim \alb_X(X) = \dim X$.
If $X$ has only canonical singularities and $K_X$ is nef, then $K_X$ is semi-ample.
In particular,
Conjecture \ref{Abund} (2) even without the extra Hyp(A) holds for varieties with maximal albanese dimension.
\end{itemize}
\end{proposition}

\begin{proof}
(1) It is well known that every Enrique surface has an elliptic fibration.
By \cite[Theorem in Appendix]{MM}, every $K3$ surface has infinitely many
singular elliptic curves.
Thus (1) follows from the classification of algebraic surfaces.

(2) Take a common resolution $Z$ of $Y$ and $A$, i.e., let
$p : Z \to A$ and $q : Z \to Y$ be two biraitonal morphisms.
Write $K_Z = p^*K_A + E_p = E_p$ where $E_p \ge 0$ is $p$-exceptional and
$\Supp E_p$ is equal to $\Exc(p)$, the exceptional locus of $p$.
Write $K_Z = q^*K_{Y} + E_1 - E_2 \sim_{\Q} E_1 - E_2$ where both $E_i \ge 0$ are $q$-exceptional
and there is no common irreducible component of $E_1$ and $E_2$.

Equating the two expressions of $K_Z$, we get
$E_1 \sim_{\Q} E_2 + E_p$. Since $E_1$ is $q$-exceptional, its
Iitaka $D$-dimension is zero, so $E_1 = E_2 + E_p$.
Thus $\Exc(p) = \Supp E_p \subseteq \Supp E_1 \subseteq \Exc(q)$.
Hence there is a birational surjective morphism $h : A \to Y$ such that $q = h \circ p$:

\vskip 2pc
\begin{center}\psset{arrows=->, colsep=2cm, arrowscale=1.5, nodesep=2mm} \everypsbox{\footnotesize}
\begin{psmatrix}
$Z$ & $A$ & $Y$
\ncline[linestyle=dashed]{1,1}{1,2}
\ncline{1,1}{1,2}
\Bput{$p$}
\ncline{1,2}{1,3}
\Bput{$h$}
\ncarc[arcangleA=30,arcangleB=30]{1,1}{1,3}
\Aput{$q$}
\end{psmatrix} .
\end{center}

\vskip 1pc \noindent

If $h : A \to Y$ is not an isomorphism, then it contracts a curve $C$ on $A$ to a point on $Y$.
Clearly, $C^2 < 0$.
By the genus formula, $2g(C) - 2 = C^2 + C . K_A = C^2 < 0$. So $C \cong \PP^1$.
This contradicts the fact that there is no rational curve on the abelian variety $A$.
Thus $h$ is an isomorphism.

(3) Since $K_Z \sim_{\Q} 0$, $Z$ is not of general type.
By (1), either $Z$ is birational to an abelian surface, or $Z$ has infinitely many rational or elliptic curves.
In the first case, $Z$ is an abelian surface by (2). Thus
$Z$ is not algebraic Lang hyperbolic in all cases.

(4) We refer to \cite[\S 3.13]{KM} for its proof or references.

(5) This follows from the base point freeness result for nef and big canonical divisors of klt varieties
(cf.~\cite[Theorem 3.3]{KM}).

(6) It is proven in \cite[Theorem 3.6]{Fj09}.
\end{proof}

The result below is just \cite[Theorem 8.3]{Ka85}; see also
\cite[Lemma 8.1]{Ka85} and \cite[Theorem 1]{Ka81} for the assertion (1).

\begin{lemma}\label{alb}
Let $X$ be a normal projective variety with only canonical singularities and $K_X \sim_{\Q} 0$.
Suppose that the irregularity $q(X) > 0$. Then we have:
\begin{itemize}
\item[(1)]
The albanese map $\alb_X : X \to A := \Alb(X)$ is a surjective morphism,
where $\dim A = q(X)$.
\item[(2)]
There is an \'etale morphism $B \to A$ from another abelian variety $B$
such that the fibre product $X \times_A B \cong Z \times B$ for some variety $Z$.
\item[(3)]
$X$ is covered by images of abelian varieties $\{z\} \times B$ ($z \in Z$).
\end{itemize}
\end{lemma}

\begin{lemma}\label{pe}
Let $X$ be a normal projective variety of dimension $n$ such that $K_X$ is $\Q$-Cartier.
Suppose that $X$ is not uniruled and $K_X \equiv 0$ (numerically).
Then $X$ has at worst canonical singularities and $K_X \sim_{\Q} 0$.
\end{lemma}

\begin{proof}
Let $\gamma : \widetilde{X} \to X$ be a projective resolution
and write $K_{\widetilde X} = \gamma^*K_{X} + (E_1 - E_2) \equiv E_1 - E_2$
such that $E_i \ge 0$ ($i = 1, 2$) are $\gamma$-exceptional and have no common components.
Since $X$ and hence $\widetilde{X}$ are non-uniruled,
$K_{\widetilde X}$ is pseudo-effective by \cite[Theorem 2.6]{BDPP}.
Let $K_{\widetilde X} = P_1 + N_1$ be the $\sigma$-decomposition in
\cite[ChIII, \S 1.b]{ZDA}, which is also called the Zariski decomposition in codimension-one.
Here $P_1$ is the movable part and
$N_1$ the negative part which is an effective divisor.
Then $E_1 \equiv P_1 + (N_1 + E_2)$.
Since RHS $- (N_1 + E_2)$ is movable, the negative part of LHS which is $E_1$,
satisfies $E_1 \le N_1 + E_2$ (cf.~\cite[ChIII, Proposition 1.14]{ZDA}).
Thus $(N_1 + E_2 - E_1)$ and also $P_1$ are pseudo-effective divisors, but their sum
is numerically equivalent to zero.
Take general members $H_i$ ($1 \le i \le n-1$) in a linear system $|H|$ with $H$ a very ample divisor
on $\widetilde{X}$.
Then
$$0 = H^{n-1} . (P_1 + N_1 + E_2 - E_1) = H^{n-1} . P_1 + H^{n-1} . (N_1 + E_2 - E_1) .$$
Hence $H^{n-1} . P_1 = 0 = H^{n-1} . (N_1 + E_2 - E_1)$. Thus
$0 = (N_1 + E_2 - E_1) \cap (H_1 \cap \cdots \cap H_{n-1})$.
Since $N_1 + E_2 - E_1$ is an effective divisor and the restriction to a subvariety of an ample divisor
is still an ample divisor, we get $N_1 + E_2 - E_1 = 0$.
Thus $N_1 + E_2 = E_1$. Since $E_i$ have no common components,
either $E_2 = 0$, or $E_1 = 0$ (and hence $N_1 = E_2 = 0$).
So $E_2 = 0$ and hence $K_{\widetilde X} = \gamma^*K_{X} + E_1$
with $E_1 \ge 0$. Therefore, $X$ has at worst canonical singularities by definition.
This together with $K_{X} \equiv 0$ imply that $K_{X} \sim_{\Q} 0$ by
\cite[Theorem 8.2]{Ka85}.
\end{proof}

\begin{lemma}\label{rnef}
Let $W$ be an algebraic Lang hyperbolic projective variety, $V$ a projective variety with at worst klt
singularities, and $h: V \to W$ a morphism such that $V \to h(V)$ is generically finite.
Assume that $K_V$ is relatively nef over $W$. Then $K_V$ is nef.
\end{lemma}

\begin{proof}
Suppose the contrary that $K_V$ is not nef and hence
there is a $K_V$-negative extremal rational curve $C$
by the cone theorem \cite[Theorem 3.7]{KM}. Since $W$ is hyperbolic
and hence contains no rational curve,
$C$ must be contracted by $V \to W$. So $K_V . C < 0$ for
a curve $C \subset V$ contracted by $V \to W$.
This contradicts the relative nefness of $K_V$ over $W$.
Hence $K_V$ is nef. This proves the lemma.
\end{proof}

\begin{lemma}\label{genfin}
Let $W$ be an algebraic Lang hyperbolic projective variety, $X$ a projective variety and
$g: X \to W$ a morphism such that $X \to g(X)$ is generically finite.

Then there is a birational map $X \dasharrow X_m$ over $W$,
i.e., there is a (generically finite) morphism $g_m : X_m \to W$ such that the natural composition
$X \dasharrow X_m \overset{g_m}\to W$ coincides with $g: X \to W$ $($and hence $g_m(X_m) = g(X))$:

\vskip 2pc
\begin{center}\psset{arrows=->, colsep=2cm, arrowscale=1.5, nodesep=1mm} \everypsbox{\footnotesize}
\begin{psmatrix}
$X$ & $X_m$ & $W$
\ncline[linestyle=dashed]{1,1}{1,2}
\ncline{1,2}{1,3}
\Bput{$g_m$}
\ncarc[arcangleA=30,arcangleB=30]{1,1}{1,3}
\Aput{$g$}
\end{psmatrix} .
\end{center}

\vskip 1pc \noindent
Further, $X_m$ has at worst canonical singularities; the canonical divisor
$K_{X_m}$ is nef; and $K_{X_m}$ is also relatively ample over $W$.
\end{lemma}

\begin{proof}
Since $g(X)$ is also hyperbolic,
replacing $W$ by $g(X)$, we may assume that $g$ is surjective (and generically finite).
Take a projective resolution $X' \to X$.
Since the relative dimension of $X'$ over $W$ is zero, the canonical divisor
$K_{X'}$ (and indeed every divisor on $X'$) is relative big over $W$.
The main Theorem 1.2 in \cite{BCHM} says that
$X'$ has a log canonical model $X_m$ over $W$, so
$X_m$ has at worst canonical singularities and $K_{X_m}$ is relative ample over $W$.
This $X_m$ is obtained from a log terminal model of $X'$ over $W$ followed by a
birational morphism over $W$ using the relative-base point freeness result for
relative nef and big divisors; see \cite[Theorem 1.2, Definition 3.6.7, Theorem 3.9.1]{BCHM}.
We note that \cite{BCHM} considers log pairs, while ours is the pure case;
so the smoothness of $X'$ implies that the log terminal (resp. log canonical) model of $X'$
over $W$ has at worst terminal (resp. canonical) singularities.
By Lemma \ref{rnef},
$K_{X_m}$ is nef. This proves the lemma.
\end{proof}

\begin{remark}\label{rgenfin}
\begin{itemize}
\item[(1)]
By the proof, every subvariety $S$ of $X$ (with $g_{| S}$ generically finite)
or of hyperbolic $W$ has a minimal model $S_m$ with only canonical singularities.

\item[(2)]
Assume $n(K_{X_m}) = \dim X_m \ge 1$ and Conjecture \ref{Abund} (2) holds.
Then the Kodaira dimension $\kappa(X_m) > 0$.
By \cite[Theorem 7.3]{Ka85b}, $K_{X_m}$ is ``good'' (or abundant).
So it is semi-ample by \cite[Theorem 1.1]{Ka85b}, which has a new proof by Fujino.

\item[(3)]
Suppose that $Y$ is a normal projective variety birational to the $X$ in Lemma \ref{genfin}
and $K_Y$ is $\Q$-Cartier. Then $K_Y$ is pseudo-effective.
Indeed, let $\sigma: Y' \to Y$ be a resolution. Since $g : X \to W$ is generically
finite and $W$ is hyperbolic, $X$ and hence $Y$ and $Y'$ are non-uniruled.
By \cite[Theorem 2.6]{BDPP}, $K_{Y'}$ is pseudo-effective.
Hence $K_Y = \sigma_*K_{Y'}$ is pseudo-effective.
\end{itemize}
\end{remark}

\section{Proof of Theorems}

In this section, we prove results in Introduction, and Theorems \ref{ThB} and \ref{ThD} which
imply Theorem \ref{ThC} and Proposition \ref{PropA}, respectively.
We also prove Corollaries \ref{Cor1} and \ref{Cor2}, all in dimension $\le 3$,
where we do not assume
Conjecture \ref{nALH} or \ref{Abund}.

When $\dim X \le 3$, Case (3) below does not occur.

\begin{theorem}\label{ThB}
Let $W$ be an algebraic Lang hyperbolic projective variety, $X$ a projective variety of dimension $n$
and $g: X \to W$ a morphism such that $X \to g(X)$ is generically finite.
Assume either $n \le 3$ or Conjecture $\ref{Abund} (2)$
$($resp.~either $n \le 3$ or Conjecture $\ref{Abund} (2)$ without the extra Hyp(A)$)$ holds
for all varieties birational to $X$.

Then there is a birational map $X \dasharrow X_m$ over $W$, i.e.,
there is a morphism $g_m : X_m \to W$ such that the composition
$X \dasharrow X_m \overset{g_m}\to W$ coincides with $g: X \to W$
$($and hence $g_m(X_m) = g(X) )$:

\vskip 2pc
\begin{center}\psset{arrows=->, colsep=2cm, arrowscale=1.5, nodesep=1mm} \everypsbox{\footnotesize}
\begin{psmatrix}
$X$ & $X_m$ & $W$
\ncline[linestyle=dashed]{1,1}{1,2}
\ncline{1,2}{1,3}
\Bput{$g_m$}
\ncarc[arcangleA=30,arcangleB=30]{1,1}{1,3}
\Aput{$g$}
\end{psmatrix} .
\end{center}

\vskip 1pc \noindent
Further, $X_m$ is a minimal variety with at worst canonical
singularities; $K_{X_m}$ is relatively ample over $W$; and
one of the following is true.
\begin{itemize}
\item[(1)]
$K_{X_m}$ is ample. Hence both $X_m$ and $X$ are of general type.
\item[(2)]
$X_m$ is an absolutely minimal
Calabi-Yau variety of dimension $n \ge 3$,
and $g_m : X_m \to g_m(X_m) = g(X) \subseteq W$ is a finite morphism.
\item[(3)]
There is an almost holomorphic map $\tau : X_m \dasharrow Y$
$($resp.~a holomorphic map $\tau : X_m \to Y )$
such that its general fibre $F$ is an absolutely minimal
Calabi-Yau variety with $3 \le \dim F < \dim X_m$,
and $(g_m)_{| F} : F \to g_m(F) \subset g_m(X_m) = g(X) \subseteq W$ is a finite morphism.
The Kodaira dimension $\kappa(X) \le \dim Y \le n-3$.
\end{itemize}
\end{theorem}

In Theorem \ref{ThD} below, Conjecture \ref{nALH} or \ref{Abund} is not assumed.
$\aaa(W)$ denotes (the Zariski-closure of) the image
$\Imm (\alb_W : W \dasharrow \Alb(W))$
of the albanese map.
Since $\Alb(W)$ is generated by $\aaa(W)$, and $\dim \Alb(W) = q(\widetilde{W}) =
\frac{1}{2}b_1(\widetilde{W})$
for any projective resolution $\widetilde{W}$ of $W$,
the condition (iii) in Theorem \ref{ThD} is satisfied if $n = 4$ and
$q(\widetilde{W}) > 0$.

For related work, the authors of \cite{HLW} also considered albanese map for smooth $W$
and used classical results of Ueno,
\cite[Theorem 1]{Ka81}, etc.,
while we use
\cite{Fj09}, \cite{Ka85}, \cite{Ka85b}.

\begin{theorem}\label{ThD}
Let $X$ be an algebraic Lang hyperbolic projective variety of dimension $n$.
Assume one of the following conditions holds.
\begin{itemize}
\item[(i)]
$X$ has maximal albanese dimension, i.e., $\dim \aaa(X) = \dim X$.
\item[(ii)]
The Kodaira dimension $\kappa(X) \ge n-3$.
\item[(iii)]
$\dim \aaa(X) \ge n-3$ and $\kappa(\aaa(X)) \ge n-4$.
\end{itemize}
Then one of the following is true.
\begin{itemize}
\item[(1)]
There is a birational surjective morphism $g_m : X_m \to X$ such that
$X_m$ has at worst canonical singularities, $K_{X_m}$ is ample and
hence both $X_m$ and $X$ are of general type.
\item[(2)]
$\kappa(X) \in \{n-3, n-4\}$, and
$X$ is covered by subvarieties whose normalizations
are absolutely minimal Calabi-Yau varieties of dimension three.
\end{itemize}
\end{theorem}

\vskip 1pc
{\it We prove Theorem $\ref{ThA}$.}
Let $f'' : X'' \to X$ be a dlt blowup with $E_{f''}$ the reduced $f''$-exceptional divisor
(cf.~\cite[Theorem 10.4]{Fj}). Namely,
$X''$ is $\Q$-factorial, $(X'', E_{f''})$ is dlt (and hence $X''$ is klt) and
\begin{equation}\label{EqC1}
(f'')^*K_X = K_{X''} + E''
\end{equation}
where $E''$ is $f''$-exceptional and satisfies $E'' \ge E_{f''}$.

Since $f''$ is birational, $K_{X''}$ is relative big over $X$.
By \cite[Theorem 1.2, Definition 3.6.7]{BCHM},
there is a birational map $\sigma: X'' \dasharrow X'$ over $X$,
such that $\sigma^{-1}$ does not contract any divisor,
$X'$, like $X''$, has only $\Q$-factorial klt singularities
and $K_{X'}$ is relatively nef over $X$ via a birational morphism $f' : X' \to X$.
Pushing forward the equality (\ref{EqC1}) above by $\sigma_*$, we get
\begin{equation}\label{EqC2}
(f')^*K_X = K_{X'} + E'
\end{equation}
where $E' := \sigma_* E'' \ge \sigma_* E_{f''} = E_{f'}$
and $E_{f'}$ is the reduced $f'$-exceptional divisor.
Since $K_{X'}$ is relatively $f'$-nef over $X$, our $K_{X'}$ is nef by Lemma \ref{rnef}:

\begin{claim}\label{nef}
$K_{X'}$ is nef.
\end{claim}

We continue the proof of Theorem \ref{ThA}.
Let $\tau: X'\dasharrow Y$ be a nef reduction of the nef divisor $K_{X'}$,
and $n(K_{X'}) := \dim Y$ the nef dimension of $K_{X'}$;
let $F$ be a general (compact) fibre of $\tau$;
then $K_F = (K_{X'})_{| F}$ is numerically trivial
(cf.~Theorem \ref{Th8aut}).

\begin{lemma}\label{I}
Assume the hypotheses of Theorem \ref{ThA}. For the $X'$ and
$\tau: X'\dasharrow Y$ defined above,
it is impossible that $\dim Y = 0$.
\end{lemma}

\begin{proof}
Consider the case $\dim Y = 0$. Then $K_{X'} \equiv 0$ (numerically zero).
Since $X$ is hyperbolic, $X$ and hence $X'$ are non-uniruled.
By Lemma \ref{pe}, $X'$ has at worst canonical singularities, and $K_{X'} \sim_{\Q} 0$;
the same hold for $X$, noting that $K_X = f'_*K_{X'} \sim_{\Q} 0$
(cf.~\cite[Corollary 1.5]{Ka88}).

We claim that $X$ is a Calabi-Yau variety. We only need to show that
the irregularity $q(X) = 0$.
Suppose the contrary that $q(X) > 0$. Then, by Lemma \ref{alb},
$X$ is covered by images of abelian varieties of dimension equal to $q(X)$.
This contradicts the hyperbolicity of $X$.
Therefore, $q(X) = 0$. Hence $X$ is a Calabi-Yau variety.
This contradicts the hyperbolicity of $X$,
Remark \ref{rThA} and the assumed Conjecture \ref{nALH}.
This proves Lemma \ref{I}.
\end{proof}

\begin{lemma}\label{II}
Assume the hypotheses of Theorem \ref{ThA}. For the $X'$ and
$\tau: X'\dasharrow Y$ defined preceding Lemma \ref{I},
it is impossible that $1 \le \dim Y < \dim X'$.
\end{lemma}

\begin{proof}
Consider the case $1 \le \dim Y < \dim X'$.
A general fibre $F$ of $\tau : X' \dasharrow Y$ satisfies
$1 \le \dim F = \dim X - \dim Y < \dim X$. Also
$K_F \equiv 0$.
Since $X$ and hence the general fibre $F$
of $\tau : X' \dasharrow Y$ are not covered by rational curves
by the hyperbolicity of $X$, $F$ is not uniruled.
By Lemma \ref{pe}, $F$ has at worst canonical singularities
and $K_F \sim_{\Q} 0$.

Factor the birational map $X' \supset F \to f'(F) \subset X$
as $F \to F^n \to f'(F)$, where $F \to F^n$ is a birational morphism and
$F^n \to f'(F)$ is the normalization.
By \cite[Corollary 1.5]{Ka88},
$F^n$ has only canonical singularities and $K_{F^n} \sim_{\Q} 0$.

If $q(F^n) > 0$, by Lemma \ref{alb},
$F^n$ and hence $f'(F)$ and $X$ are covered by images of abelian varieties
of dimension equal to $q(F^n)$, contradicting the hyperbolicity of $X$.
Thus $q(F^n) = 0$, so $F^n$ is a Calabi-Yau variety.
By the assumed Conjecture \ref{nALH} and Remark \ref{rThA}, there is a non-constant
holomorphic map $V \to F^n$ from an abelian variety $V$, which, combined with the (birational and)
finite map $F^n \to f'(F)$, produces a non-constant holomorphic map $V \to X$,
contradicting the hyperbolicity of $X$.
This proves Lemma \ref{II}.
\end{proof}

By the two lemmas above, we are left with the case
$\dim Y = \dim X'$. Namely,
the nef dimension $n(K_{X'}) = \dim X'$.
By the assumed abundance Conjecture \ref{Abund} (1),
$K_{X'}$ is semi-ample. Hence
$\Phi_{|sK_{X'}|}$, for some $s > 0$, is a morphism onto a normal variety $X_c$, with connected fibres,
and there is an ample $\Q$-divisor $H_c$ on $X_c$ such that $K_{X'} \sim_{\Q} \Phi_{|sK_{X'}|}^*H_c$.
Clearly, this map which is now holomorphic,
is (up to birational equivalence) a nef reduction of $K_{X'}$ and
also denoted as $\tau : X' \to X_c$. In other words, $Y = X_c$, $K_{X'}$
is big (and nef), and $\tau$ is birational.
Pushing forward the equality $K_{X'} \sim_{\Q} \tau^*H_c$
by $\tau_*$, we get $K_{X_c} \sim_{\Q} H_c$
and hence $K_{X'} \sim_{\Q} \tau^*K_{X_c}$
(so that $\tau$ is a crepant birational morphism) with $K_{X_c}$ an ample $\Q$-divisor.
Since $X'$ is klt and $\tau$ is crepant, $X$ is also klt.
By \cite[Corollary 1.5]{HM}, every fibre of $\tau : X' \to X_c$ is rationally chain connected
and hence is contracted to a point by
the birational morphism $f' : X ' \to X$ due to the hyperbolicity of $X$
and the absence of rational curves on $X$.
Thus $f'$ factors through $\tau$, i.e., there is a birational
morphism $f_c : X_c \to X$ such that $f' = f_c \circ \tau$
(cf.~\cite[Proof of Lemma 14]{Ka81}).
In summary, we have the following commutative diagram:

\par \vskip 2pc

\begin{center}\psset{arrows=->, colsep=2cm, rowsep=1cm, arrowscale=1.5, nodesep=2mm} \everypsbox{\footnotesize}
\begin{psmatrix}
$X''$ & $X'$ & $X_c$ \\
$X$ & $X$ & $X$
\ncline[linestyle=dashed]{1,1}{1,2}
\Aput{$\sigma$}
\ncline{1,2}{1,3}
\Aput{$\phi_{|sK_{X'}|}$}
\ncline[arrows=-, doubleline=true, nodesep=4mm]{2,1}{2,2}
\ncline[arrows=-, doubleline=true, nodesep=4mm]{2,2}{2,3}
\ncline{1,1}{2,1}
\Bput{$f''$}
\ncline{1,2}{2,2}
\Bput{$f'$}
\ncline{1,3}{2,3}
\Bput{$f_c$}
\end{psmatrix}
\end{center}

\par \vskip 1pc
Pushing forward the equality (\ref{EqC2}) above by $\tau_*$, we get
$$ f_c^*K_X = K_{X_c} + E_c .$$
Here $E_c := \tau_* E' \ge \tau_* E_{f'} = E_{f_c}$  and $E_{f_c}$ is the reduced
$f_c$-exceptional divisor.
So the image $f_c(E_c)$ is contained in $\Nklt(X)$, the non-klt locus of $X$,
which is a Zariski-closed subset of $X$ consisting of exactly the non-klt points of $X$.

{\it This proves Theorem $\ref{ThA}$; see Remark \ref{rThA} (1) for the final part.}

\par \vskip 1pc
If we do not assume Conjecture \ref{nALH} in Theorem \ref{ThA},
the argument above actually shows:

\begin{remark}
Let $X$ be a $\Q$-Gorenstein normal projective variety of dimension $n$ which is algebraic Lang hyperbolic.
Assume either $n \le 3$ or Conjecture $\ref{Abund} (1)$ holds for all varieties of dimension $\le n$. Then
either $K_X$ is ample at smooth points and klt points of $X$ as
detailed in Theorem $\ref{ThA}$;
or $X$ is a Calabi-Yau variety of dimension $n \ge 3$; or
$X$ is covered by subvarieties $\{F_t'\}$ whose normalizations
are Calabi-Yau varieties of the same dimension $k$ with $3 \le k < n$.
\end{remark}

\par \vskip 1pc
Since Conjecture \ref{nALH} is true for surfaces,
and the abundance Conjecture \ref{Abund} is known in dimension $\le 3$, we can and will soon prove
the following consequences of Theorem \ref{ThA}.

\begin{corollary}\label{Cor1}
Let $X$ be a $\Q$-Gorenstein normal projective surface which is algebraic Lang hyperbolic.
Then $X$ is of general type and the canonical divisor $K_X$ is ample.
\end{corollary}

We can not remove the second alternative below even when $X$ is smooth,
because, for instance, we do not know, at the moment, the non-hyperbolicity
of a general smooth Calabi-Yau
$n$-fold and a Hyperk\"ahler $n$-fold when $n > 2$.

\begin{corollary}\label{Cor2}
Let $X$ be a $\Q$-Gorenstein normal projective threefold which is algebraic Lang hyperbolic.
Then either
the canonical divisor $K_X$ is ample
at the smooth points and klt points of $X$;
or $X$ is a Calabi-Yau variety.
\end{corollary}

\begin{setup}\label{PfCor1-2}
{\bf Proof of Corollaries \ref{Cor1} and \ref{Cor2}}
\end{setup}

We use the fact that Conjecture \ref{nALH} holds in dimension $\le 2$,
and both Conjectures \ref{Abund} (1) and (2) (even without the extra Hyp(A))
hold in dimension $\le 3$ (cf.~Proposition \ref{solconj}).

Corollary \ref{Cor1} is a consequence of Theorem \ref{ThA} and the observation:
if $f_c : X_c \to X$ is a birational morphism between $\Q$-Gorenstein normal projective surfaces
and $K_{X_c}$ is ample, then $K_X = (f_c)_*K_{X_c}$ is also ample,
by using Nakai-Moishezon ampleness criterion and the projection formula.

For Corollary \ref{Cor2}, we follow the argument for the proof of Theorem \ref{ThA}.
Thus we have to consider Cases (I) $\dim Y = 0$, (II) $0 < \dim Y < \dim X$  and (III) $\dim Y = \dim X$.
In Case (I), $X$ has been proven to be a Calabi-Yau variety of dimension three; for this purpose,
Conjecture \ref{nALH} was not used. In Case (II), a contradiction has been deduced utilizing
Conjecture \ref{nALH} in dimension $\le 2$.
In Case (III), using the proven abundance conjecture in dimension $3$, we get the first possibility
in the conclusion part of Corollary \ref{Cor2}. This proves Corollary \ref{Cor2}.

\begin{setup}
{\bf Proof of Theorems \ref{ThC} and \ref{ThB}}
\end{setup}

If we let $g : X \to W$ in Theorem \ref{ThB} be the identity map $\id_X : X \to X$,
we get Theorem \ref{ThC}; we also use the observation that
a birational finite morphism from a normal variety like $X_m$ or $F$
in Theorem \ref{ThB} is just the normalization map.
Thus we have only to (and are going to) prove Theorem \ref{ThB}.

We may assume Conjecture \ref{Abund} (2) (the case without the extra Hyp(A)
is similar and indeed easier); for varieties of dimension $\le 3$, this assumption is automatically
satisfied by Proposition \ref{solconj}.

Theorem \ref{ThB} is clearly true when $\dim X = 1$.
So we may assume that $n = \dim X \ge 2$.
We apply Lemma \ref{genfin} and
let the birational map $X \dasharrow X_m$ over $W$ and $g_m : X_m \to W$ be as there,
where $X_m$ has only canonical singularities and $K_{X_m}$ is relatively ample over $W$
and is also nef.
Since $g : X \to W$ is generically finite, so is $g_m : X_m \to W$.

Let $\tau : X_m \dasharrow Y$ be a nef reduction of the nef divisor $K_{X_m}$.
For our $\tau : X_m \dasharrow Y$ and $g_m : X_m \to W$ here (with $W$ algebraic Lang hyperbolic)
we closely follow the argument of Theorem \ref{ThA} for $\tau : X' \dasharrow Y$ and $f' : X' \to X$ there
(with $X$ algebraic Lang hyperbolic),
but we do not assume Conjecture \ref{nALH} here.

Suppose that $\dim Y = 0$, i.e., $n(K_{X_m}) = 0$, or $K_{X_m} \equiv 0$.
Now Lemma \ref{I} is applicable under the current weaker assumption.
Precisely, due to the lack of the assumption of Conjecture \ref{nALH} here,
instead of the contradiction there, we have that
$X_m$ is a Calabi-Yau variety.
The relative ampleness of $K_{X_m}$ over $W$ implies that $g_m : X_m \to g_m(X_m) = g(X) \subseteq W$ is a finite morphism.
This and the hyperbolicity of $W$ imply that $X_m$ is hyperbolic.
Further, $X_m$ is absolutely minimal (cf.~Remark \ref{rThA}).
By Proposition \ref{solconj}, $\dim X_m \ge 3$.
So Case \ref{ThB} (2) occurs.

Suppose that $1 \le n(K_{X_m}) < n$. Denote by $F$ a general fibre of $\tau : X_m \dasharrow Y$.
For this case, Lemma \ref{II} is applicable even under the current weaker assumption. So
$F$ is a Calabi-Yau variety.
The map
$(g_m)_{|F} : F \to g_m(F) \subset g_m(X_m) = g(X) \subseteq W$ is a finite morphism,
otherwise, a curve $C$ in $F$ is $g_m$-exceptional and hence $K_{X_m} . C > 0$,
by the relative ampleness of $K_{X_m}$ over $W$, contradicting the numerical triviality
of $(K_{X_m})_{|F}$ entailing $K_{X_m} . C = 0$.
This and the hyperbolicity of $W$ imply that $F$ is hyperbolic.
Further, $F$ is absolutely minimal (cf.~Remark \ref{rThA}).
By Proposition \ref{solconj}, $\dim F \ge 3$.
So Case \ref{ThB} (3) occurs. Indeed, for the final part,
the well known Iitaka addition for Kodaira dimension
implies that $\kappa(X) \le \kappa(F) + \dim Y = \dim Y = n - \dim F \le n-3$.

Suppose that $n(K_{X_m}) = n$.
Then $K_{X_m}$ is semi-ample by Conjecture \ref{Abund} (2) (cf.~Remark \ref{rgenfin}).
Hence $K_{X_m}$ is nef and big, and there is a birational morphism
$\gamma : X_m \to Z$ onto a normal variety $Z$ such that
$K_{X_m} \sim_{\Q} \gamma^*H$ for an ample $\Q$-divisor $H$ on $Z$.
Hence both $X_m$ and $X$ are of general type.
Since $\gamma$ is birational, the projection formula implies that $K_Z = \gamma_*K_{X_m} \sim_{\Q} H$. Hence
$K_Z$ is ample and $K_{X_m} \sim_{\Q} \gamma^*K_Z$. Thus $Z$, like $X_m$, has only canonical singularities.
If $\gamma$ is not an isomorphism, then it has a positive-dimensional fibre.
By \cite[Corollary 1.5]{HM}, every fibre of $\gamma$ is rationally chain connected.
So we may assume that $\gamma$ contracts a rational curve $C$ on $X_m$ to a point on $Z$.
Since $W$ is hyperbolic, $g_m : X_m \to X \subseteq W$ contracts the rational curve $C$ on $X_m$ to a point on $W$.
This and $C . K_{X_m} = C . \gamma^*K_Z = \gamma_*C . K_Z = 0$
contradict the relative ampleness of $K_{X_m}$ over $W$.
So $\gamma : X_m \to Z$ is an isomorphism. Hence $K_{X_m}$ is ample.
Thus Case \ref{ThB} (1) occurs.
This proves Theorem \ref{ThB}.

\begin{setup}
{\bf Proof of Corollary \ref{Cor3}}
\end{setup}

For the assertion (1), we apply Theorem \ref{ThB} to the inclusion map $g: X \hookrightarrow W$
for a projective subvariety $X$ of $W$.
By Theorem \ref{ThB}, either Case \ref{ThB} (1) occurs and hence $X$ is of general type,
or Case \ref{ThB} (2) or (3) occurs. We use the the notation there:
birational morphism $g_m: X_m \to X$, etc.
If Case \ref{ThB} (2) (resp.~(3)) occurs,
by the assumed Conjecture \ref{nALH} and Remark \ref{rThA},
there is a non-constant holomorphic map
from an abelian variety $V$ to $X_m$ (resp.~to $F$),
thus, combined with the (birational and) finite morphism
$$g_m : X_m \to g_m(X_m) = g(X) = X \subseteq W$$
(resp.~$(g_m)_{| F} : F \to g_m(F) \subset g_m(X_m) = g(X) = X \subseteq W$), this map gives a
non-constant holomorphic map from $V$ to $W$,
contradicting the hyperbolicity of $W$.
This proves the assertion (1).

For the assertion (2), the case $\dim W = 1$ is clear.
We may assume that $\dim W \ge 2$.
We apply Theorem \ref{ThB} and let $g : X \to W$ be
the identity map $\id_X : X \to X$, with $W = X$.
Hence there is a birational morphism $g_m: X_m \to X$ such that
\ref{ThB} (1), (2) or (3) occurs.
We use the following known fact (cf.~\cite[Lemma 8.1]{Ka85}, or Lemma \ref{alb}):

\begin{fact}
If a projective variety $V$ has only canonical $($or more generally rational$)$ singularities,
then the albanese map $\alb_V : V \to \Alb(V)$ is a well defined morphism
and $\dim \Alb(V) = q(V) = h^1(V, \OO_V)$.
\end{fact}

In Case \ref{ThB} (1), $X$ is of general type.

In Case \ref{ThB} (2), the irregularity $q(X_m) = 0$,
and hence $\Alb(X) = \Alb(X_m)$ is a point.
Thus $X$ itself is the fibre of $\alb_X$. Hence $\dim X \le 2$,
by the assumption. Thus $\dim X = 2$, by the extra assumption that $X = W$ has dimension at least two.
By Proposition \ref{solconj}, either $X$ has infinitely many rational or elliptic curves,
or it is birational to an abelian surface (and hence $q(X_m) = 2$), or it is of general type.
Since $X = W$ is hyperbolic and $q(X_m) = 0$, $X = W$ is of general type.

In Case \ref{ThB} (3), by the above fact, $\alb_{X_m} : X_m \to \Alb(X_m) = \Alb(X)$
is a well defined morphism. Since
the general fibre $F$ of the nef reduction map $\tau : X_m \dasharrow Y$ which is almost holomorphic,
is a Calabi-Yau variety, we have $q(F) = 0$, so $\Alb(F)$ is a point. By the universality of
the albanese map, the composition $F \hookrightarrow X_m \to \Alb(X_m)$
factors through $F \to \Alb(F)$. So $\alb_{X_m} : X_m \to \Alb(X_m)$
maps $F$ to a point. Hence $g_m(F) \subset X$ is contained in a
general fibre $G$ of $\alb_X : X \overset{g_m^{-1}}\dasharrow X_m \to \Alb(X_m) = \Alb(X)$.
Indeed, $\alb_X$ factors as
$$X \overset{g_m^{-1}}\dasharrow X_m \overset{\tau}\dasharrow Y \overset{\eta}\dasharrow \Alb(X_m) = \Alb(X)$$
for some rational map $\eta$, by applying \cite[Lemma 14]{Ka81} to the domain of the almost holomorphic map $\tau$.
Now $\dim G \ge \dim g_m(F) = \dim F \ge 3$,
with the last inequality shown in Theorem \ref{ThB}.
This contradicts the assumption.
This proves Corollary \ref{Cor3}.

\begin{setup}
{\bf Proof of Proposition \ref{PropA} and Theorem \ref{ThD}}
\end{setup}

We prove Theorem \ref{ThD}.
We first assume that $X$ has maximal albanese dimension.
Since abundance conjecture holds for varieties with maximal albanese dimension (cf.~Proposition \ref{solconj}),
we can apply Theorem \ref{ThC} to our $X$. So Case \ref{ThC} (1), (2) or (3) occurs.
We use the notation $g_m : X_m \to X$ there, where $K_{X_m}$ is relatively ample over $X$, and is also nef.
By \cite[Theorem 3.6]{Fj09}, $K_{X_m}$ is semi-ample.
So there exist a holomorphic map $\tau := \Phi_{|sK_{X_m}|} : X_m \to Y$ onto a normal variety $Y$,
for some $s > 0$, and an ample divisor $H$ on $Y$
such that $K_{X_m} \sim_{\Q} \tau^*H$. We can take this $\tau$ as a nef reduction of $K_{X_m}$ in
Theorem \ref{ThC}. Let $F$ be a general fibre of $\tau$.

If Case \ref{ThC} (1) occurs, we are in Case \ref{ThD} (1).

Consider Case \ref{ThC} (3). Then $(g_m)_{|F} : F \to g_m(F)$ is the
normalization map, and $F \subseteq X_m$ is a Calabi-Yau variety of dimension $\ge 3$.
Since $X$ and hence $X_m$ have maximal albanese dimension, i.e., $\dim \aaa(X_m) = \dim X_m$,
we may assume that the restriction $(\alb_{X_m})_{|F}$ to a general fibre $F$ of $\tau$,
is a generically finite morphism
onto the image $F_a \subseteq \Alb(X)$.
It is known that $F_a$, being a subvariety of an abelian variety, satisfies $\kappa(F_a) \ge 0$
with equality holding only when $F_a$ is the translation of a subtorus.
Now $0 = \kappa(F) \ge \kappa(F_a) \ge 0$. Thus $\kappa(F_a) = 0$ and hence $F_a$ is an abelian variety.
Therefore, $0 = q(F) \ge q(F_a) = \dim F_a = \dim F \ge 3$. This is a contradiction.

If Case \ref{ThC} (2) occurs, we get a similar contradiction by the arguments above, with $F = X_m$.

Next we assume that $\kappa(X) \ge n-3$.
Let $I_X: X \dasharrow Y$ be the Iitaka fibration
with $F$ a very general fibre.
So $\kappa(F) = 0$ and $\dim F = \dim X - \kappa(X) \le 3$ by the assumption.
Since $F$ is a subvariety of the hyperbolic variety $X$, it is also hyperbolic.

If $\dim F = 0$, then $X$ is of general type.
Since abundance conjecture holds for varieties of general type (cf.~Proposition \ref{solconj}),
we can apply Theorem \ref{ThC} to our $X$, and only Case \ref{ThC} (1) there occurs.
This fits Case \ref{ThD} (1).

Consider the case $\dim F \in \{1, 2, 3\}$.
Applying Theorem \ref{ThC} to the hyperbolic $F$ and noting that $\kappa(F) = 0$ and $1 \le \dim F \le 3$,
only Case \ref{ThC} (2) there occurs for $F$:
the normalization of $F$ is an absolutely minimal Calabi-Yau variety
of dimension three, so $\kappa(X) = n - \dim F = n - 3$; also these $F$ cover $X$.
This fits Case \ref{ThD} (2).

Finally, we assume that
$\dim \aaa(X) \ge n-3$ and $\kappa(\aaa(X)) \ge n-4$.
By the results obtained so far, we may add the extra assumptions:
$n > \dim \aaa(X)$ and $\kappa(X) \le n-4$.
Let $G$ be an irreducible component of a general fibre of the albanese map
$\alb_X : X \dasharrow \aaa(X) \subseteq \Alb(X)$.
Then $1 \le g:= \dim G = \dim X - \dim \aaa(X) \le 3$.
The hyperbolicity of $X$ implies that $X$ and hence $G$ are not uniruled.
So $G$ has a good minimal model in the sense of Kawamata \cite{Ka85b},
by the abundance theorem in dimension $\le 3$ (cf.~\cite[\S 3.13]{KM} or Proposition \ref{solconj}).
In particular, $\kappa(G) \ge 0$.
By Iitaka's $C_{n, n-g}$ proved in
\cite[Corollary 1.2]{Ka85}, $\kappa(X) \ge \kappa(G) + \kappa(\aaa(X)) \ge 0 + 0$.
By the assumptions,
$$n-4 \ge \kappa(X) \ge \kappa(G) + \kappa(\aaa(X)) \ge \kappa(G) + n-4 \ge n-4. $$
Thus the inequalities above all become equalities.
In particular, $\kappa(G) = 0$ and $\kappa(X) = n-4$.
Since $\kappa(X) \ge 0$, we have $n \ge 4$.
As in the previous paragraph, applying Theorem \ref{ThC} to $G$,
the normalization of $G$ is an absolutely minimal Calabi-Yau variety
of dimension three; also these $G$ cover $X$.
This fits Case \ref{ThD} (2).

This proves Theorem \ref{ThD}.

\par \vskip 1pc
Now we prove Proposition \ref{PropA}. Set $n := \dim X$.
Each of the three conditions in Proposition \ref{PropA} implies one of the first two conditions in
Theorem \ref{ThD}. Hence we can apply Theorem \ref{ThD}.
We may assume that Case \ref{ThD} (2) occurs.
Thus $X$ is covered by subvarieties $F_t$ whose
normalizations $F_t^n$ are absolutely minimal
Calabi-Yau varieties of dimension three.
The algebraic Lang hyperbolicity of $X$ implies the same for $F_t$ and also $F_t^n$.
This contradicts
the assumed Conjecture \ref{nALH} in dimension three (cf.~Remark \ref{rThA}).

This proves Proposition \ref{PropA}.

%
%
%
%


\end{document}